\documentclass[a4paper,12pt]{amsart}

\usepackage[cp1251]{inputenc} 
\usepackage[T1]{fontenc}     
\usepackage[ngerman,english]{babel} 

 \usepackage{pdfpages}  
 
\usepackage{amsmath, amsfonts, amssymb, amsthm, bbm}

\usepackage{graphicx}
\usepackage{tikz}
\usepackage{pdfpages}
\usepackage{floatflt}
\usepackage{multicol}
\usepackage{xcolor}
\usepackage{caption}
\usepackage{subcaption}

\usepackage[margin=1.2in]{geometry}

\newtheorem{theorem}{Theorem}
\newtheorem{lemma}{Lemma}[section]
\newtheorem{proposition}{Proposition}[section]
\newtheorem{definition}{Definition}[section]
\newtheorem{remark}{Remark}[section]
\newtheorem{cor}{Corollary}[section]
\numberwithin{equation}{section}
 
\usepackage{comment}

\usepackage{hyperref}
\hypersetup{
    colorlinks=true,
    linkcolor=red,
    filecolor=gray,      
    urlcolor=cyan,
    citecolor=gray
}

\newcommand{\R}{{\mathbb{R}}}
\newcommand{\Sph}{{\mathbb{S}}}

\newcommand{\N}{{\mathbb{N}}}
\newcommand{\Z}{{\mathbb{Z}}}

\renewcommand{\d}{\partial}

 \newcommand{\F}{{\mathcal{F}}}

\newcommand{\e}{\varepsilon}

\newcommand{\V}{{\mathcal{V}}}
\newcommand{\U}{{\mathcal{U}}}
\renewcommand{\H}{{\mathcal{H}}}

\title{Submetry onto one-dimensional space}
\author{Darya Sukhorebska}
\address{\parbox{\linewidth}{Institut f\"ur Mathematik, Universit\"at Paderborn, 
Warburger Str. 100, 33098 Paderborn, Germany  \\
B.Verkin Institute for Low Temperature Physics and Engineering of the National Academy of Sciences of Ukraine, 
47 Nauky Ave., Kharkiv, 61103, Ukraine}} \email{darya.sukhorebska@uni-paderborn.de}
\date{}

\thanks{The author was partially supported by Deutsche Forschungsgemeinschaft
(DFG, German Research Foundation) under the project number 281869850. }


\begin{document}

 \begin{abstract} 
We provide the full classification of equidistant decompositions of 
a two-dimensional Euclidean plane and a two-dimensional sphere.
\end{abstract} 

 \maketitle
 
\section*{Introduction}
A decomposition of a smooth complete Riemannian manifold $M$ into a
collection of pairwise equidistant closed leaves $L_y$, $y\in Y$,
is in  one-to-one correspondence to the natural quotient map $P:M\to Y$, 
where $Y$ is a space of leaves. 
In the most general case, when the leaves of the decomposition may
have different dimension and perhaps of regularity less than $C^{1,1}$,
the map $P:M\to Y$ is called \textit{submetry}.
Submetries were first introduced by Berestovskii \cite{Ber87}
and defined as a map,
that sends balls of a given radius in $M$ surjectively onto balls of the same radius in $Y$.
If, however,  all the leaves $L_y$ are $C^{1,1}$ but still of different dimension, 
the map $P:M\to Y$ is called a  \textit{transnormal} (or Riemannian) submetry.
If $Y$ is a smooth complete Riemannian manifold,
then the map $P$ is a Riemannian submersion \cite{BerGui}. 
The quotient map   $P:M\to M/G$ of the Riemannian manifold $M$
associated to an isometric group action $G$ provides canonical examples of 
 submetries that are not necessarily submersions.

The structure of submetries
was the primary focus of the investigation in \cite{KL20}. 
 One of their central results establishes that the fibers $L_y$ 
are sets of \textit{positive reach} $r>0$. 
By definition, it means that the closest foot-point projection on $L_y$
is uniquely defined within a neighborhood $U(L_y) \subset M$.
In this case, the distance function $d_{L_y}:U(L_y)\to [0, r)$ is a submetry itself.

Even though the sets of positive reach were studied in many works 
\cite{Fed59}, \cite{Lyt05}, \cite{Lyt23}, \cite{RatZaj17},
examples of low regularity are hard to construct explicitly.

The aim of this paper is to present 
 the full classification of equidistant decompositions 
of the two-dimensional Euclidean space and spherical space,
supplementing the known examples.

Before we state our first result, 
let us introduce the following curve in~$\R^2$.
Let $\Pi_a$ and $\Pi_{-a}$ be two disjoint closed half-planes. 
Denote $l_a:=\d \Pi_a$, $l_{-a}:=\d \Pi_{-a}$ 
and $h:=dist(l_a, l_{-a})\geq 0$.
Fix the point  $x_0\in l_a$ and choose
$y_0\in l_{-a}$ such that the distance between $y_0$ and
the orthogonal projection of $x_0$ to $l_{-a}$ is $2a$. 
Consider two families of concentric half-circles 
$$\F_a = \Pi_a\cap\{S(x_0, r_i)\}_{i=0}^\infty  \,\,\, \textnormal{and} \,\,\,
 \F_{-a} = \Pi_{-a}\cap\{S(y_0, r_i)\}_{i=0}^\infty $$ 
 where $S(x, r)$ is a circle of radius $r$ around $x \in \R^2$ and 
 $r_i=(1+2i)\cdot a$ for any $i=0,1,2,\dots$.  
Let $\F_h$ be a family of line segments in
$\R^2\backslash \left(\Pi_a \cup \Pi_{-a}\right)$ 
orthogonal to $l_{\pm a}$ connecting pairwise the
endpoints $l_a \cap \F_a$ and $l_{-a} \cap \F_{-a}$.
From the choice of  $x_0$, $y_0$ and $r_i$, we obtain that the union
$$\sigma_{a,h}:= \F_h \cup \F_a \cup \F_{-a} $$
is a $C^{1,1}$ curve on $\R^2$ for any real number $h\geq 0$ and 
$a > 0$ (see Figure~\ref{fig:R2intro}).

\begin{theorem}\label{thm:1}
    Let $P:\R^2\to Y$ be a submetry with connected fibers
    between a two-dimensional Euclidean space $\R^2$ and
    a one-dimensional Alexandrov space~$Y$. 
    Then up to isometry one of the following holds
   \begin{enumerate}
       \item[i)]  $Y=\R$ and $P:\R^2\to \R$ is an orthogonal projection; 
       \item[ii)] $Y=[0, \infty)$
              and $P$ is a distance function to 
              $L=P^{-1}(0)$ and $L$ is either a point, a closed straight line segment or a half line;
       \item[iii)] $Y=[-a,a]$ and 
                $P$ is a signed distance function to the curve $P^{-1}(0)=\sigma_{a,h}$
                for any $h\geq 0$.
   \end{enumerate}
\end{theorem} 
The first two statements are not new. In fact, 
  $i)$ follows from more general result, established in~\cite{BerGui}
  and~$ii)$ follows from Lemma~5.2 and Example~6.6  in~\cite{KL20}.
  The third statement will be proved as Theorem~\ref{thm:R2_segm} below.\\

\begin{figure}[h]
  \begin{minipage}[b]{0.45\textwidth}
        \includegraphics[width=1.\linewidth]{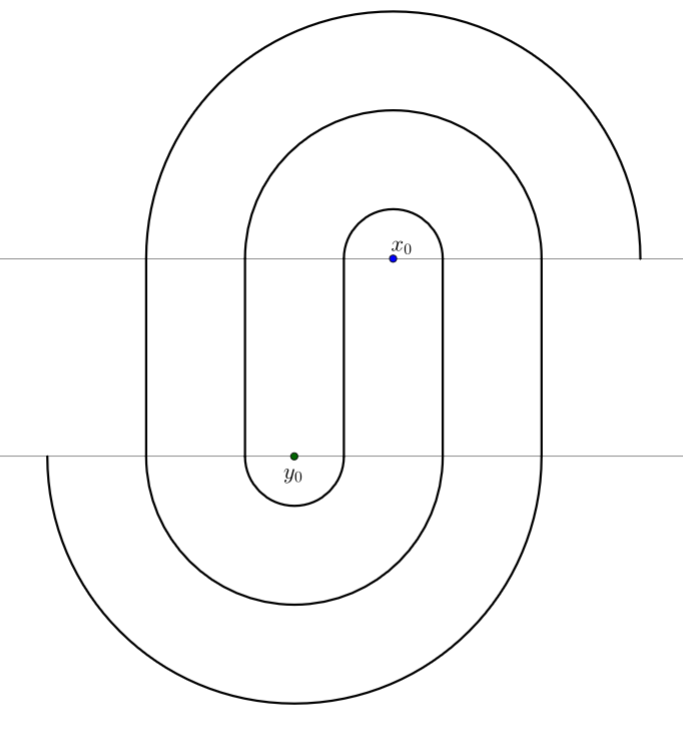}
        \caption{$\sigma_{a,h} \subset \R^2$.}
        \label{fig:R2intro}
  \end{minipage}
    \hfill
  \begin{minipage}[b]{0.45\textwidth}
        \includegraphics[width=1.\linewidth]{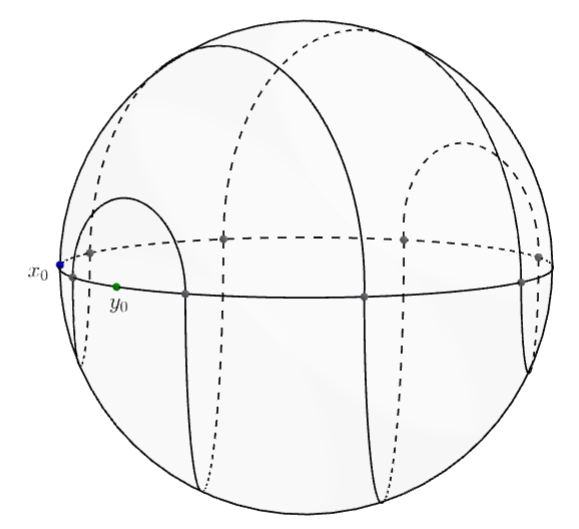}
        \caption{$\sigma_{s} \subset \Sph^2$. }
        \label{fig:S2intro}
  \end{minipage}      
\end{figure}

We can construct a curve similar to $\sigma_{a,h}$ in $\Sph^2$. 
The great circle $S^1\subset \Sph^2$ divides the sphere in two closed
hemispheres $\Sph^2_1$ and $\Sph^2_{-1}$ sharing $S^1$ as a boundary.
Denote $a:=\pi/2k$ for $k>1$ and let $s$ be a number coprime with $k$.
Fix the points $x_0, y_0 \in S_1$ at distance $2a\cdot s$ 
between them. 
Consider two families of concentric half-circles 
$$\F_1 = \Sph^2_1\cap\{S(x_0, r_i)\}_{i=0}^{k-1}  \,\,\, \textnormal{and} \,\,\,
 \F_{-1} = \Sph^2_{-1}\cap\{S(y_0, r_i)\}_{i=0}^{k-1} $$ 
 where $S(x, r)$ is the set of points at distance $r$ from $x$ on $\Sph^2$ and 
 $r_i=(1+2i)\cdot a$ for any $i=0,1,..., k-1$.  
By the choice of $r_i$ and $s$ 
the union 
$$\sigma_{s}=\F_1 \cup \F_{-1} $$
forms a connected $C^{1,1}$ curve on $\Sph^2$
(see Figure~\ref{fig:S2intro}).
\begin{theorem}\label{thm:2}
Let $P:\Sph^2\to Y$ be a submetry with connected leaves between
 a two-dimensional spherical space $\Sph^2$ and a one-dimensional Alexandrov space~$Y$.\\
Then the base space $Y$ is a segment of length $\pi/k$ and 
      \begin{enumerate}
          \item[a)]  If $k=1$, 
                then $P$ is a quotient map of an isometric group action of $S^1$ on $\Sph^2$.
          \item[b)] If $k>1$, then, up to isometry,
                $P$ is the signed distance function to the curve~$P^{-1}(0)=\sigma_{s}$ 
                for some $s$ coprime with $k$.
      \end{enumerate}
\end{theorem}
When $k=1$, the submetry $P:\Sph^2 \to Y^1$ is transnormal.
In this case, the result coincides with the analogous result for 
singular Riemannian foliations on $\Sph^2$ due to Radeschi \cite{Rad12}.

By \cite{Lyt01}, any submetry $P:X \to Y$ between Alexandrov spaces admits 
a canonical factorization $P=P_1\circ P_0$, where 
the submetry $P_0: X\to Y_0$ has connected fibers
and the submetry $P_1:Y_0 \to Y$ has discrete fibers. 
All one-dimensional Alexandrov spaces
are the quotients of the real line  $Y=\R/G$ by a discrete group of isometries $G$. 
From the result of Lange \cite[Theorem 1.2]{Lange20}
follows that 
all submetries $P_1:Y_0 \to Y$ with discrete fibers are coverings. 
Namely, $Y_0 =\R/G_0$, where $G_0$ is a subgroup of $G$ 
(see Subsection~\ref{subsec:discrete}). 
These results in combination with our Theorems \ref{thm:1} and \ref{thm:2}
gives us the full classification of the equidistant decompositions 
of the two-dimensional Euclidean space and spherical space.\\
\par


This paper is organized as follows:
in Section~\ref{sec:prelim} we give all necessary definitions and preliminaries, 
in Section~\ref{sec:prep} we do further preparations.
In Section~\ref{sec:sub_R2}  we first summarize and present
all submetries from $\R^n$ to $\R$ and $[0, \infty)$ in Theorem~\ref{thm:Rn_sub},
 thereby  proving $i)$ and $ii)$ of Theorem~\ref{thm:1}.
 Theorem~\ref{thm:R2_segm} corresponds to the statement~$iii)$ of Theorem~\ref{thm:1}.
In the last Section \ref{sec:sub_S2}, we prove Theorem~\ref{thm:2}.

\subsection*{Acknowledgments.} The author is grateful to Alexander Lytchak  
for suggesting this problem and for his continuous support throughout the project.\\

\section{Preliminaries}\label{sec:prelim}

\subsection{Definition and basic properties}\label{subsec:def}
A submetry $P:X \rightarrow Y$ is a map between metric spaces $X$ and $Y$
which sends balls $B_r(x)$ of radius $r$ around any point $x\in X$  
surjectively onto the closed $r$-ball around $P(x)$, i.e.
$ P(B_r(x))=B_r(P(x))$.
The space $X$ is called total space, $Y$ is called base space.

Let $d_A:X\to [0, \infty)$ be a distance function to a subset $A$ in $X$. 
The following property follows from the definition of  submetry. 
\begin{proposition}[Lemma 2.1., \cite{KL20}]\label{basic}
   A map $P: X \rightarrow Y$  between metric spaces is a submetry if and only if $P$
   is surjective and, for any $y \in Y$, we have $d_y \circ P = d_{P^{-1}(y)}$. 
\end{proposition}

In particular, it means that the submetry $P:X \rightarrow Y$ 
induces an equidistant decomposition of the space $X$.
Indeed, for any two points $y_1, y_2 \in Y$ 
the distance between fibers
$P^{-1}(y_1)$ and $P^{-1}(y_2)$ is constantly equal to $d(y_1, y_2)$ on $Y$.
The opposite is also true.
In fact, if a metric space $X$ admits a decomposition 
into a collection of pairwise equidistant leaves $L_y$, $y\in Y$,
then the set of leaves $Y$ with induced distance
becomes a metric space. 
Then the canonical projection $P:X\to Y$ is a submetry.

The submetry $P:X\to Y$ induces some properties of the total space $X$ onto $Y$ 
(see \cite[Proposition 1]{Ber87}). 
In particular, if $X$ is a smooth complete Riemannian manifold $M$,
then $Y$ is a complete and proper metric space.
We may also replace $Y$ with its given metric by 
 $P(M)$ with an induced length metric. 
Due to \cite[Corollary 2.10]{KL20} $P:M\to Y$ remains a submetry.
Thus, without loss of generality, we assume that $P:M\to Y$
is surjective and $Y$ is a length space.

The restriction of a submetry $P:X \rightarrow Y$ to 
an open subset $A \subset X$ might not be a submetry, but it is a local submetry.
The definition of local submetry was given in \cite[Definition 2.7]{KL20}.
However, for us there is no need to localize the definition, since
 by \cite[Corollary 2.9]{KL20} a local submetry $P:M\to Y$ 
 from a complete Riemannian manifold $M$ is a submetry.

Additionally, 
if $P:M\to Y$ is a submetry from the Riemannian manifold~$M$,
then $Y$ is also a locally contractible. 
In fact, by \cite[Corollary 1.5]{KL20}, 
the tangent space $T_yY$ has a canonical decomposition 
$T_yY=\R^l\times T_y^0Y$, for some $l\geq 0$,
where $T_y^0Y$ is the Euclidean cone over 
an Alexandrov space of diameter at most $\pi/2$.
The exponential map $exp_y$ defines a homeomorphism from the 
$r$-ball around the origin in $T_yY$ to the $r$-ball in $Y$. 
Since the ball  around the origin in  $T_yY$ is contractible, 
so is the ball in $Y$.

\subsection{Horizontal lift}\label{subsec:lift}

Given a submetry $P:M\to Y$,
a rectifiable curve $\gamma:[a,b]\to M$ is called \textit{horizontal}
if the length of $\gamma$ in $M$ equals the length of $P\circ \gamma$ in $Y$. 

Recall that $M$ is a smooth Riemannian manifold 
and $Y$ is a proper length space. 
Then, from \cite[Lemma 2.8]{KL20}
it follows that for any rectifiable curve $\sigma: [a,b] \to Y$ 
starting at $\sigma(a) = y\in Y$
there exists a horizontal curve $\gamma:[a,b]\to M$ 
such that  $\gamma(a)\in P^{-1}(y)$ and $\sigma = P \circ \gamma$.
The curve $\gamma$ is then called a \textit{horizontal lift} of $\sigma$. 
For a locally Lipchitz curves  it is equivalent to 
the property that the velocities $|\gamma'(t)|$ and $|\sigma'(t)|$ 
coincide for all $t\in [a,b]$. 

Following \cite{KL20}, we call a curve $\gamma:[0,s] \to M$ 
a \textit{$P$-minimal geodesic} if  
it is parametrized by arclength and
$$ dist\left( P\left(\gamma(0) \right), P \left(\gamma(s) \right) \right) = s.$$
The $P$-minimal geodesic $\gamma:[0,s]\to M$, if it exists,
 realizes the distance between the fibers through its endpoints. 
If $\gamma$ is a $P$-minimal geodesic, 
then $\gamma$ is a geodesic in $M$ and
$P\circ \gamma$  is a geodesic in $Y$. 
Moreover, any horizontal lift
of a geodesic $\sigma:[a,b]\to Y$  is a $P$-minimal geodesic. 
Using this statement  in combination with Proposition~\ref{basic}, we get the following.
\begin{lemma}\label{lem:dist}
    Let $M$ be a complete Riemannian manifold. 
    \begin{enumerate}
        \item If $P:M\to [a,\infty)$ is a submetry onto a half line,
    then $P$ is a distance function to the set $P^{-1}(a)$.
        \item  If $P:M\to [a,b]$ is a submetry onto a segment, 
    then $P$ is a distance function to the set $P^{-1}(a)$ (or to $P{-1}(b)$),
    or, equivalently, $P$ is a signed distance to the middle fiber $P^{-1}(\frac{b-a}{2})$.
    \end{enumerate}

\end{lemma}

However, if $\gamma$ is a horizontal geodesic in $M$, 
then its image $P\circ\gamma $ 
is not necessarily a geodesic in $Y$, but a 
quasigeodesic (see \cite[Proposition 4]{GW11}). 
Moreover, they have a special form
\begin{proposition}[Corollary 7.4, \cite{KL20}]\label{quasigeod}
    Let $P:X \to Y$ be a local submetry between Alexandrov regions.
    Let $\gamma:[0,l] \to X$ be a geodesic. 
    Then $P\circ\gamma$ is a quasigeodesic and 
    there is a partition $0\leq t_1 < t_2 < \dots, < t_N\leq l$ 
    such that each restriction $\left(P\circ\gamma\right)|_{[t_i, t_{i+1}]}$
    is a locally minimizing geodesic. 
\end{proposition}

\subsection{Sets of positive reach}\label{subsec:pos_reach}

 The following key result was proved in the work by Kapovitch and Lytchak.
\begin{proposition}[Theorem 1.1 \cite{KL20}]\label{fibers_pr}
Let $M$ be a Riemannian manifold and $P:M \to Y$ be a submetry.
Then any fiber $L=P^{-1}(y)$, $y\in Y$ of the submetry $P$
 is a set of positive reach in $M$.
\end{proposition}
We recall that a subset $L$ of a Riemannian manifold $M$ 
 has \textit{positive reach}
 if the closest point projection on $L$ is 
 uniquely defined in a neighborhood $U(L)$ in $M$. 
 Moreover, the following is also true
\begin{proposition}[Proposition 6.3 \cite{KL20}]\label{pr_subm}
Let $L\subset M$ be a closed subset of positive reach,
nowhere dense in a Riemannian manifold $M$.
Then $d_L:U(L) \to [0, \infty)$ is a local submetry with $L=P^{-1}(0)$.
\end{proposition}

The structure of subsets of positive reach is well understood
and it does not depend on the Riemannian metric,
but only on the $C^{1,1}$-atlas.
Such sets were investigated
in particular in \cite{Fed59}, \cite{Ban82},
\cite{Lyt04_pr}, \cite{Lyt05},
\cite{RatZaj17}, \cite{Lyt23}.

Notice also that 
the topological dimension $\dim L$ of a set of positive reach $L$
coincides with its Hausdorff-dimension \cite{Fed59}.
The set of positive reach $L$ in $M$
has a well-defined \textit{tangent cone} $T_xL$ 
that is a convex subset of $T_xM$
for all $x\in L$~\cite[Theorem 4.8]{Fed59}
and  $\dim L$ is the maximum of the 
dimensions of convex cones $T_xL$.

Let $L_{reg}$ be the set of all points $x\in L$ such that 
the tangent cone $T_xL$ is isometric to $\R^{\dim L}$.
Then, by \cite[Theorem 7.5]{RatZaj17} the subset $L_{reg}$
is not-empty, open in $L$ and it is a $C^{1,1}$-submanifold of $M$. 
In case $L\backslash L_{reg}$ is empty, 
$L$ is a $C^{1,1}$-manifold without boundary. 
Moreover:
\begin{proposition}[Proposition 1.4., \cite{Lyt05}]\label{prop:pr_reg}
    Let $L$ be a set of positive reach in $M$. 
    Then the following are equivalent:
    \begin{itemize}
        \item[1)] The set $L$ is a $C^{1,1}$-submanifold;
        \item[2)] The set $L$ is a topological manifold;
        \item[3)] All tangent spaces $T_xL$ for $x \in L$ are Euclidean spaces. 
    \end{itemize}
\end{proposition}

If  $L\backslash L_{reg}$ is not empty,  then
$L\backslash L_{reg}$ is locally closed subset of positive reach 
of dimension $\leq \dim L-1$. 
The following Lemma is a direct corollary of \cite[Theorem D]{RatZaj17} 
and \cite[Theorem 1.1]{Lyt23} adapted for our case.
\begin{lemma}\label{Lem:1dim_pr}
    Let $L\neq \emptyset$ be a connected 
    set of positive reach and empty interior in $\R^2$ or $\Sph^2$. 
    Then $L$ is either a point or a $C^{1,1}$ curve 
    (one-dimensional manifold) with or without boundary.    
\end{lemma}

\section{Properties and preparations}\label{sec:prep}

\subsection{Transnormal regions}\label{subsec:tr_reg}
Let $P:M\to Y$ be a submetry. 
For any point $x\in M$ consider the fiber $L=P^{-1}(P(x))$ passing through $x$.
 As we said before, the fiber $L$ 
 has a well defined tangent cone $T_xL$. 
The \textit{normal (or horizontal) cone} 
$N_xL$ is the polar cone of $T_xL$ in $T_xM$.
In other words, $N_xL$ is the set of all vectors in $T_xM$ 
enclosing angles $\geq \pi/2$ with all vectors in $T_xL$. 

For any point $x\in L$ and unit vector $h\in N_xL$,
the geodesic $\gamma_{h}:[0,l)\to M$ starting in $x$
in the direction $h$ and parametrized by arclength
satisfies 
\begin{equation}\label{eq:hor_geod}
    dist \left(L, \gamma_h(s) \right)=s,
\end{equation}
for all $s$ such that  $\gamma_h(s)\in U(L)$. 
Recall that $U(L)$ is the maximal open set in 
which $L$ has positive reach.

Note that the equation (\ref{eq:hor_geod}) does not guarantee
that the geodesic $\gamma_h$ starting in the horizontal direction
$h\in N_xL$ is itself a horizontal curve for any $s \in (0, l)$. 
However, it is true locally, meaning that for any $x\in M$
there is $r_x\in (0, l)$, such that 
the geodesic $\gamma_h:[0, r_x) \to M$ is a horizontal curve
and a $P$-minimal geodesic (see \cite[Proposition 7.3]{KL20}).   

We can identify those regions in $M$ 
on which $\gamma_h$ restricts to a global horizontal curve. Namely,
\begin{definition}\label{def:tran_region}
    Let $P:M\to Y$ be a submetry and
    let $\V$ be an open set in~$M$. 
    Let $\gamma_h:[0,l)\to M$ be a geodesic emanating at $x \in \V$ 
    in the horizontal direction $h\in N_xL$.  
    Assume  $\gamma_h$ is perpendicular to any fiber it intersects
    for all times when $\gamma_h$ is in $\V$.     
    Then the set $\V$ is called a \textbf{transnormal region} of the
    manifold $M$. 
\end{definition}
If $\V=M$, then the submetry $P:M\to Y$ is called \textit{transnormal}.

\begin{lemma}\label{lem:tran_region}
    Let $P:M\to Y$ be a submetry 
    and let $L=P^{-1}(P(x))$ denote the fiber through any point $x\in M$.
    Then an open set $\V \subset M$ is a transnormal region of $M$ 
    if and only if 
     $L\cap \V \subset L_{reg}$ for any $x\in \V$.    
\end{lemma}
Recall  that $L_{reg}$ is the set of all points $x\in L$ such that 
the tangent cone $T_xL$ is isometric to $\R^{\dim L}$,
and $L_{reg}$  is a $C^{1,1}$-submanifold of $M$. 

The proof of Lemma~\ref{lem:tran_region} almost word-for-word repeats 
a the part of the proof of \cite[Proposition 12.5]{KL20}.
For completeness, we include it here as well.

\begin{proof}[Proof of Lemma~\ref{lem:tran_region}]
    To prove the forward direction, we proceed by contradiction.
    Let $\V \subset M$ be a transnormal region of the submetry $P:M\to Y$. 
    Assume now, that there is a point $x\in\V$ and 
    a fiber $L$ through $x$ such that $L\cap \V \not\subset  L_{reg}$. 
    It means, there is a point $y\in L\cap \V$ such that 
    the tangent space $T_yL$ is not a vector space. 
    Therefore, the normal space $N_xL$ is also not a vector space. 
    Hence, we can find a unit vector $h\in N_yL$ such 
    that $-h \not\in N_yL$. 
    Since $\V \subset M$ is an open set, then there is a small 
    neighborhood $\U$ of the point $y$ such that $\U \subset \V$. 
    Hence, there is a  small $\e >0$ such that 
    the geodesics $\gamma_h:[0,\e] \to M$ and $\gamma_{-h}:[0,\e] \to M$
    are both in $\U$.
    However,  $\gamma_h:[0,\e] \to M$ is the horizontal geodesic,
    while $\gamma_{-h}:[0,\e] \to M$ is not.
    If we then take a vector $v=-\gamma'_h(t)$ for $t\in (0, \e)$,
    then the geodesic in the direction $v$ does not stay perpendicular
    for all fibers it intersect inside~$\V$, a contradiction. 

    The sufficiency can also be shown by contradiction.
    Assume  $\gamma_h:[0, l)\to M$ is the geodesic
    starting at the point $x\in \V$ in the horizontal direction $h$
    and lies inside $\V$ for all times $[0, l)$.
    By \cite[Proposition 7.3]{KL20}, there exist $r\in (0,l]$
    such that $\gamma_h:[0,r] \to M$ is a horizontal curve. 
    Assume $r$ is  maximal with this property and $r<l$.
    Then $\gamma_h$ meets the fiber $L_r:=(P^{-1}\circ P)(\gamma_h(r))$
    in the horizontal direction $-\gamma'_h(r) \in N_{\gamma_h(r)} L_r$.
    Moreover, the normal space $N_{\gamma_h(r)} L_r$ is a vector space. 
    Hence, the opposite direction $\gamma'_h(r)$ is also horizontal.
    Since the point $\gamma_h(r) \in \V$ and $\V$ is open, then 
    there is a small $\e>0$ such that the restriction 
    $\gamma_h: [r, r+\e) \to M$ is also a horizontal curve. 
    Since $r<l$, we get a contradiction to the minimality of $r$. 
    Hence, $\V$ is a transnormal region. 
\end{proof}

 The horizontal geodesic $\gamma_h$ leaves the transnormal region $\V$
 the moment it meets a fiber $L$ outside of its regular part, 
 i.e. when $\gamma_h(t) \in L \backslash L_{reg}$. 
 If this is the case, then the following result holds.
\begin{proposition}[Theorem 12.4 \cite{KL20}]\label{non_mnfd_f}
    Let $P:M\to Y$ be a local submetry.
    If a fiber $L=P^{-1}(y)$ is not a  $C^{1,1}$-submanifold of $M$ 
    then $y\in \d Y$.
\end{proposition}

By \cite[Proposition 12.7]{KL20}
 transnormal submetries have the \textit{equifocality} property.
It means the following.
Let $P:M\to Y$ be a transnormal submetry and let 
$\gamma_1, \gamma_2:[0,s)\to M$ be two horizontal geodesics
starting at two different points $\gamma_1(0)\neq \gamma_2(0)$ on a fiber $L$.
If $P\circ \gamma_i$ coincide on $[0, \e)$ for some $\e \in (0, s)$, 
then $P\circ \gamma_i$  coincide along the whole segment $[0, s)$.
The same is true  when both geodesics $\gamma_1$ and $\gamma_2$ 
lie  in the transnormal region $\V$ of $M$. 

In this work, we are mostly interested in the case 
when $Y$ is a one-dimensional manifold. 
Here, equifocality for transnormal regions is clear.

\subsection{Submetry with connected leaves}\label{subsec:connected}

When the map  $P:M\to Y$ is a fiber bundle, 
then using the long exact sequence of homotopy groups 
one can deduce that
if the total space $M$ is simply connected and
all fibers are connected manifolds,
then the base space $Y$ is also simply connected. 
A similar statement is true when the map $P:M\to Y$ is a submetry.

\begin{lemma}\label{lem:smpl_con}
   Let $P:M\to Y$ be a submetry between  a complete, 
   simply-connected Riemannian manifold $M$
   and a connected metric space $Y$. 
   If all fibers of $P$ are connected, then $Y$ is also simply connected. 
\end{lemma}

\begin{proof}
     Consider a loop $\tilde{c}(t):[0,l]\to Y$ 
     such that $\tilde{c}(0)=\tilde{c}(l)=p$.
     Since $Y$ is a locally contractible space, 
     $\tilde{c}(t)$ can be approximated by a piecewise geodesic, 
     closed loop $c(t):[0,b]\to Y$ homotopic to $\tilde{c}(t)$ and $c(0)=c(b)=p$. 
     The curve $c(t)$ is rectifiable and 
     therefore, by \cite[Lemma 2.8]{KL20},
     there exists a horizontal lift 
    $\gamma(t):[0, b]\to M$ of the curve  $c(t)$. 
    The points $\gamma(0)$ and $\gamma(b)$ are on the same fiber $L=P^{-1}(p)$ but they may not coincide. 
    Since $L$ is connected, then  there is a path on $L$ 
    connecting points  $\gamma(0)$ and $\gamma(b)$. 
    Hence, the curve $\gamma(t)$ can be extended to the loop
    $\tilde{\gamma}:[0,a] \to M$ where
    $\tilde{\gamma}(0)=\tilde{\gamma}(a)=q$.
    Since $M$ is simply connected,
    then there is a continuous homotopy map 
    $F(t,s):[0,a] \times [0,1] \to M$ such that 
    $F(t,0)=\tilde{\gamma}(t)$ and $F(t,1)=\{q\}$, i.e. 
    $F$ contracts the loop $\tilde{\gamma}$ 
    to the point $q$ within $M$.
    
   The map $P\circ F(t, \cdot)$ brings the loop from the total space
    to the base space for any $s \in[0,1]$ .    
    Hence, this map $P\circ F(t, \cdot)$ is the homotopy between the loop $c_p(t)$ and the point $P(q)$. 
    Therefore, the space $Y$ is also simply connected. 
\end{proof}

Necessary for us is the following corollary of Lemma~\ref{lem:smpl_con}.

\begin{cor}\label{cor:1}
    There is no submetry $P:M\to S^1$ with connected fibers 
    when $M$ is an Euclidean space $\R^n$ or spherical space $\Sph^n$ of  dimension $n\geq 2$. 
\end{cor}

\subsection{Submetry with discrete leaves}\label{subsec:discrete}

In this section, 
we present all submetries
$P:Y_0 \to Y$ with discrete fibers 
between the one-dimensional Alexandrov spaces $Y_0$ and $Y$.
In general, Alexandrov spaces are a special class of length spaces. 
We refer the reader to \cite{BBI01} for the definition and the background. 
The one-dimensional Alexandrov spaces are easy to describe.
Namely, they are one of the following: 
circle $S^1$, real line $\R$, half line $[0, \infty)$ or 
closed interval $[a,b]$.

\begin{lemma}[Corollary of Theorem 1.2~\cite{Lange20}]\label{lem:discr_fib}
    Let $P:Y_0 \to Y$ be a submetry between one-dimensional Alexandrov spaces 
    and let $P$ have discrete fibers.
    Then one of the following holds.
\begin{enumerate}
    \item If $Y_0=\R$, then $P$ is either a covering map onto $Y=\Sph^1$
            or $P$ is a folding map onto half-line $Y=[0, +\infty)$ or onto 
            an interval of the length~$a$.     
    \item If $Y_0=[0, +\infty)$, then $P$ is a folding map  
            onto an interval of length~$a$.
    \item If $Y_0$ is an interval of length~$1$,
            then  $P$ is a folding map onto 
            an interval of length~$1/k$, $k\in \N$.
    \item If $Y_0$ is a unit circle, then 
             $P$ is a covering map either onto 
             the circle of length~$1/k$, $k\in \N$,
             or onto a segment of length~$1/2k$, $k\in \N$.
\end{enumerate}
\end{lemma}

\begin{proof}
All one-dimensional Alexandrov spaces 
are  quotients of the real line $\R$, namely $Y=\R/G$,
where $G$ is a discrete group of isometries. 
\begin{enumerate}
    \item[i)] If $G=id$ is a trivial group, then $Y=\R$.
    \item[ii)] If $G=T_{na}\cong \Z$ is a cyclic group generated by 
    the translation $T_a=\{x \to x+a\}$, $a\in \R$, $n\in \Z$,
    then $Y=\R/G$ is a circle of length $a$. 
    \item[iii)] If $G$ is a reflection $R_b=\{x\to2b-x\}$ around a point $b\in \R$,
    then $Y=\R/ G$ is a half line $[b, +\infty)$. 
    \item [iv)] If $G$ is a combination of a translation $T_a$ and $R_b$, $a<b$, 
    then $Y=\R/ G$ is a segment  $[a, b]$. 
\end{enumerate}
From \cite[Theorem 1.2]{Lange20} follows that
the submetry $P:Y_0 \to Y$ with discrete fibers are orbifold covering (in sense of Thurston). 
It means that $Y_0=\R/ G_0$ where $G_0$ is a subgroup of $G$.
This gives us a complete classification of such submetries $P:Y_0 \to Y$.

\end{proof}

\section{Proof of Theorem 1}\label{sec:sub_R2}

To prove Theorem~\ref{thm:1},
we construct submetries for different cases of the base space $Y$: 
circle $S^1$, real line $\R$, half line $[0, \infty)$ or closed interval $[a,b]$.

From Corollary~\ref{cor:1} it is clear that $Y\neq S^1$. 
The first result holds for the Euclidean space $\R^n$ of any dimension $n\geq 2$.
\begin{theorem}\label{thm:Rn_sub}
    Let $P:\R^n\to Y$ be a submetry with connected leaves between an Euclidean space $\R^n$
    and a one-dimensional space $Y$.
    The following holds up to isometry.
   \begin{enumerate}
       \item[a)] If $Y=\R$, then all fibers $P^{-1}(y)$ are totally geodesic submanifolds $\R^{n-1}$ 
       and $P: \R^{n-1}\times \R \to \R$ is an orthogonal projection;
       \item[b)] If $Y=[0, \infty)$, then 
       $P$ is a distance function to the fiber $P^{-1}(0)$
       and  $P^{-1}(0)$ is a closed convex and nowhere dense subset of $\R^n$.
   \end{enumerate}
\end{theorem} 
\begin{proof}
 a) If $Y=\R$, then by Theorem A in \cite{BerGui}, 
 $P:\R^n\to \R$ is a Riemannian submersion. 
 From Theorem B of the same work \cite{BerGui} follows, that 
 there is a totally geodesic hypersurface $\R^{n-1}$  
 such that $\R^n= \R^{n-1}\times \R$ and $P$ can be identified with 
 a projection onto the second factor.
 
 b) By Lemma~\ref{lem:dist}, 
 $P:\R^n\to[0, \infty)$ is a distance function to the fiber $P^{-1}(0)$. 
 By Proposition~\ref{fibers_pr}, the fiber $L=P^{-1}(0)$ is a set of reach $\infty$. 
 By  \cite[Theorem 4.8]{Fed59} $L$ is exactly a closed convex set nowhere dense in $\R^n$. 
\end{proof}

 \begin{remark}\label{rem:R}
     One can also prove b) of Theorem~\ref{thm:Rn_sub} using a) and 
     Theorem~A in \cite{BerGui}. 
     In fact, we can see $P:\R^n\to \R$ as a signed distance function 
     to the fiber $L=P^{-1}(y)$ for some $y\in \R$.
      This is possible since $P$ is a Riemannian submersion  and all fibers are hypersurfaces. 
      Thus, $L$ has reach $\infty$ from both sides, and 
      as a result, $L$ should be convex hypersurface from  both sides.
      Hence,  $L$ is isometric to a flat hypersurface $\R^{n-1}$.      
 \end{remark} 

 \begin{remark}
     When $n=2$, then a closed convex and nowhere dense subset $P^{-1}(0) \subset \R^2$
     is one of the following: a point, a straight line segment, a half line or a straight line.
     The last one can not occur as a fiber $P^{-1}(0)$ of the submetry
     $P:\R^2\to [0, \infty)$, since otherwise all other leaves are not connected. 
 \end{remark}

It remains to construct the submetry for the case  $Y=[a,b]$. 
Now we assume that the dimension of the ambient space is $n=2$. 
Without loss of generality, we can also assume  $Y=[-1,1]$.
Otherwise, we can rescale the Euclidean metric on $\R^n$ by the factor $l/2$,
where $l=b-a$ is the length of the segment. 

Consider the following $C^{1,1}$ curve $\sigma_h$  in $\R^2$. 
Let $\Pi_1$ and $\Pi_{-1}$ be two disjoint closed half-planes at a distance $h \geq 0$
and let  $l_1:=\d \Pi_1$, $l_{-1}:=\d \Pi_{-1}$.
Fix points  $x_0\in l_1$  and $y_0\in l_{-1}$ such that 
$dist(x_0, y_0)=\sqrt{4+h^2}$.
Let $\F_1$ and $\F_{-1}$ be the following families of concentric half-circles 
$$\F_1 = \Pi_1\cap\{S(x_0, r_i)\}_{i=0}^\infty  \,\,\, \textnormal{and} \,\,\,
 \F_{-1} = \Pi_{-1}\cap\{S(y_0, r_i)\}_{i=0}^\infty $$ 
 where $S(x, r)$ is a circle of radius $r$ around $x$ and 
 $r_i=(1+2i)$, $i=0,1,2,\dots$.  
Let $\F_h$ be a family of line segments in
$\R^2\backslash \left(\Pi_1 \cup \Pi_{-1}\right)$ 
orthogonal to $l_{\pm1}$ connecting pairwise the
endpoints $l_1 \cap \F_1$ and $l_{-1} \cap \F_{-1}$.
The choice of  $x_0$, $y_0$ and $r_i$ guarantees that the union
\begin{equation}\label{eq:sigma_h}
    \sigma_h := \F_h \cup \F_1 \cup \F_{-1} 
\end{equation}
is indeed a $C^{1,1}$ curve on $\R^2$ for any $h\geq 0$.

 \begin{theorem}\label{thm:R2_segm}
      Let $P:\R^2\to[-1, 1]$ be a submetry with connected fibers.
    Then, up to isometry, 
    the map $P$ is the signed distance function to the curve $P^{-1}(0)=\sigma_{h}$
    defined in (\ref{eq:sigma_h})  for some $h\geq 0$.
 \end{theorem}
\begin{proof}
By Lemma~\ref{lem:dist} the function $P$ can be identify with the signed distance to the 
middle fiber $L_0=P^{-1}(0)$.
By Propositions~\ref{non_mnfd_f}, $L_0$ is a $C^{1,1}$ curve without boundary in $\R^2$.
It means, that in each point $x \in L_0$ the tangent space $T_xL_0$ is 
a one-dimensional Euclidean space.
The normal space $N_xL_0$ then is also then a one-dimensional vector space, 
orthogonal to $T_xL_0$.

To construct $L_0$ we first want to understand the structure of the boundary fibers 
$L_{1}=P^{-1}(1)$ and $L_{-1}=P^{-1}(-1)$.

We claim  that the sets $L_{1}$ and $L_{-1}$
each is a $C^{1,1}$ curve with a single boundary point. 
In fact, by Lemma~\ref{Lem:1dim_pr}, $L_{\pm1}$ is either a point or 
a $C^{1,1}$ curve with or without boundary.
Consider  the sets $D_{\pm1}$ of all 
the points in $\R^2$ at a distance $\leq 1$ from $L_{\pm1}$.
These two closed sets $D_1$ and $D_{-1}$, when glued along the boundary,
cover the entire $\R^2$.
If $L_{1}$ were a point or $L_{1}$ had two boundary points, then 
$D_1$ would be a compact domain, and $L_0$ would be a closed curve. 
Consequently, $D_{-1}$ would also be a compact domain, 
which contradicts $D_1 \cup D_{-1} = \R^2$. 
If $L_{1}$ were homeomorphic to $\R$, 
the boundary of the domain $D_1$ would split in two components. 
This contradicts the fact that all fibers are connected. 
Therefore, the only possibility is that
$L_{1}$ and $L_{-1}$ are $C^{1,1}$-equivalent to half-lines
with the boundary points $p_1$ and $p_{-1}$ respectively.

By Lemma~\ref{lem:tran_region}, 
the open set $\V:=\R^2\backslash\{p_1, p_{-1}\}$ is
a transnormal region of the ambient space $\R^2$. 

We will denote by $\gamma(t):[0, \infty) \to \R^2$ the geodesic emanating from $x\in L_0$
in the direction $h_x\in N_xL$, i.e.$\gamma(t)=x+t\,h_x$.
Since $-h_x\in N_xL$, then the geodesic in the opposite direction 
can be written as $\gamma(-t)=x-t\,h_x$, $t\geq 0$.
The geodesics $\gamma(t)$ and $\gamma(-t)$ are horizontal inside $\V$.

The projection $(P\circ\gamma)(t)$ is a quasigeodesic in $Y$ for all times 
of its existence $t\in (0, t_0)$.
Moreover, by Proposition~\ref{quasigeod}, there is a partition of $(0, t_0)$  such that
\begin{equation}\label{eq:quasigeod_R2}
    (P\circ\gamma)\left(2k+(-1)^k\e \right) =\e, \,\,\, \e \in [-1, 1]
\end{equation}
when $k\in \N$ such that $2k+(-1)^k\e \leq t_0$.
The same holds in the opposite direction $-h_x$, taking $-t$ and $-\e$ instead. \\
\par

\textbf{Step 1:}\textit{ Infinitely transnormal regions}.\par
We call a connected subset $\H \subset \V $ \textit{an infinitely transnormal region} 
if for any point $x\in \H$ the horizontal geodesic 
$\gamma(t)$ lies inside $\H$ for all $t\in \R$, 
i.e. $\gamma(t)$ never intersect the point $p_1$ or $p_{-1}$ in both directions.
If $\H \neq \emptyset$, denote 
\begin{equation}\label{eq:Fh}
    \F_{h} := L_0\cap \H.
\end{equation}
From the equation~(\ref{eq:quasigeod_R2}), 
$\F_h$ is an infinite family of segments of $L_0$,
any two of which are at distance $2$ from each other. 
Each of these segments is the set of reach $\infty$ inside $\H\subset \R^2$.
Since it is true for both directions of $\gamma(t)$,
from \cite[Theorem 4.8]{Fed59} we can deduce that $\F_h$ is a family of straight line segments. 
Moreover, intersection of any fiber $L_y$, $y\in [-1,1]$, with $\H$ is an infinite family of 
straight line segments. 

The geodesics $\gamma(t)$ are then parallel straight lines in $\R^2$
and the region $\H$ is an infinite strip of length $h\geq 0$ in $\R^2$.

If there are several infinitely transnormal region $\H_i$, $i=1,2 \dots$,
then they are disjoint infinite strips enclosed by the parallel straight lines 
and the intersection of the fiber $L_0$  with any $\H_i$
is a family of straight line segments (see Figure~\ref{fig:inf_trans}). \\
\par

\begin{figure}[h]
\begin{subfigure}{0.45\textwidth}
\includegraphics[width=0.9\linewidth]{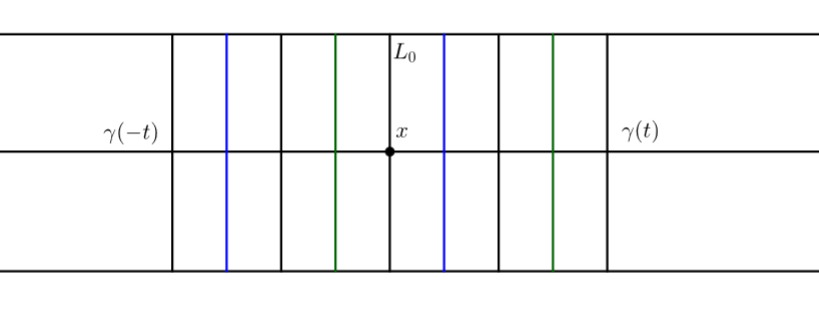}
\caption{Foliation of $\H$.}
\label{fig:inf_trans}
\end{subfigure}
\begin{subfigure}{0.45\textwidth}
\includegraphics[width=0.9\linewidth]{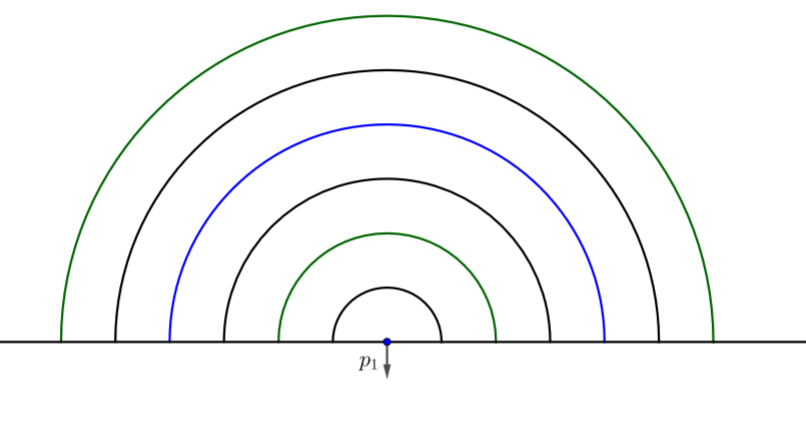}
\caption{Foliation of $\Pi_1$. }
\label{fig:Pi_1}
\end{subfigure}
\caption{Fibers: $L_0$ (black curve), $L_1$ (blue curve), and $L_{-1}$ (green curve) }
\label{fig:image2}
\end{figure}

\textbf{Step 2:} \textit{Remaining regions}. \par
For any point $x\in \V \backslash \H$ there is 
 a time $t_x\in \R$ such that the horizontal geodesic 
 $\gamma(t)$ lies inside $ \V \backslash \H$ when $0\leq t < t_x$ 
 and $\gamma(t_x) = p_1$ or~$p_{-1}$.

To construct $\V \backslash \H$ we first 
analyze the structure of the boundary fibers $L_{1}$ and $L_{-1}$ around $p_{1}$ 
and $p_{-1}$ respectively. 
The tangent space $T_{p_1}L_1$ is a Euclidean half-line, 
meaning that $T_{p_1}L_1$ contains only one unit tangent vector~$\tau_1$. 
We recall that the normal space $N_{p_1}L_1$ is the set of all vectors in $T_{p_1}\R^2$ 
enclosing angles $\geq \pi/2$ with vectors in $T_{p_1}L_1$. 
In our case, $N_{p_1}L_1$ is a two-dimensional half-plane. 
Under the exponential map of $T_{p_1}\R^2$ to $\R^2$, the 
subset $N_{p_1}L_1$ corresponds to the closed half-plane $\Pi_1$.
Denote $l_1:=\d\Pi_1$.

Since the fiber $L_0$ lies at distance $1$ from $L_1$,
then  in $\Pi_1$, $L_0$ forms a half-circle $s(p_1,1)$
of radius $1$ centered at $p_1$.
For any point $x\in s(p_1,1)$, 
the horizontal geodesic $\gamma(t)$ satisfies the following:
\begin{itemize}
    \item $\gamma(1) = p_1$;
    \item $\gamma(t) \subset \V \backslash \H$, for $t\in [0, 1)$.
\end{itemize}
We want to show now  that in the opposite direction 
$\gamma(t) \in \V \backslash \H$ for all $t\leq 0$ and  $x\in s(p_1,1)$. 
Once established, it would imply that $\Pi_1\backslash p_1 \subset \V \backslash \H$.

From the equation~(\ref{eq:quasigeod_R2}) we can deduce that inside $\Pi_1$,
$\gamma(-1-4k) \in L_{-1}$   for any $k\in \N\cup\{0\}$ and any $x\in s(p_1,1)$.
Since $s(p_1,1)$ is a half-circle, then for any fixed $k\in \N$
the points $\gamma(-1-4k)$ form a half-circle $s(p_1,2+4k)$ of  radius $2+4k$ and center $p_1$. 
Therefore, the half-plane $\Pi_1$ does not contain the singular point $p_{-1}$
and hence $\Pi_1\backslash \{p_1\} \subset \V \backslash \H$.

From~(\ref{eq:quasigeod_R2}), we can also deduce that 
the fiber $L_0$ inside $\Pi_1$ forms the family of half-circles  
\begin{equation}\label{eq:F1}
    \F_{1} = L_0\cap \Pi_{1}=  \{s(p_{1},1+2k)\}_{k=0}^\infty 
\end{equation}
where $s(p_{1},1+2k)$ is a half-circle of the radius $1+2k$ around $p_{1}$
(see Figure~\ref{fig:Pi_1}).   

In the same way, the point $p_{-1}$ of $L_{-1}$ defines 
the closed half-plane $\Pi_{-1}$, 
where $l_{-1}:=\d\Pi_{-1}$ and $l_{-1}$ is orthogonal to the 
tangent vector $\tau_{-1}$ to $L_{-1}$ at $p_{-1}$.

Since the fiber $L_0$ lies at the distance $1$ from $L_{-1}$,
 there is  a half-circle $s(p_{-1},1)$
of radius $1$ with the center at $p_{-1}$ and $ s(p_{-1},1) \subset L_0\cap \Pi_{-1}$. 
Then, similarly to the above, 
from the equation~(\ref{eq:quasigeod_R2}) we can deduce,
that intersection $L_1\cap \Pi_{-1}$ is a family of 
half-circles around $p_{-1}$ with the radius $2+4k$, $k\in \N\cup\{0\}$.
This in particular means that $\Pi_{-1}$ does not contain $p_1$ and 
$\Pi_{-1} \backslash\{p_{-1}\}\subset \V \backslash \H$.

The intersection $L_0 \cap \Pi_{-1}$ is a family of disjoint half-circles  
\begin{equation}\label{eq:F-1}
    \F_{-1} = L_0\cap \Pi_{-1}=  \{s(p_{-1},1+2k)\}_{k=0}^\infty.
\end{equation}

\textbf{Step 3:} \textit{Gluing}. \par
Summarizing what we have shown so far:
the transnormal region $\V=\R^2\backslash\{p_1, p_{-1}\}$ is a union 
of  two half-planes $\Pi_1\backslash \{p_1\}$ and $\Pi_{-1}\backslash \{p_{-1}\}$
and infinite transnormal regions $\H_1, \dots, \H_s$. 

The interiors of $\Pi_1$ and $\Pi_{-1}$ are disjoint.
It follows from the fact that $p_{-1}\notin \Pi_1 \backslash l_1$ and 
$p_1\notin \Pi_{-1}\backslash l_{-1}$
and the half-circles of the family $\F_{1}$ and $\F_{-1}$
  are orthogonal to the boundary lines $l_1$ and $l_{-1}$ respectively. 

If $l:=l_1 =l_{-1}$, then there is no room for the regions $\H_1, \dots, \H_s$. 
The line $l$ contains both $p_{1}$ and $p_{-1}$ and divides $\R^2$
into two parts $\Pi_1$ and $\Pi_{-1}$.
The union $\F_{1} \cup  \F_{-1}$ forms a connected  fiber $L_0$
if and only if $dist(p_1, p_{-1})=2$. 
Then $L_0=\F_{1} \cup  \F_{-1}$ is the curve $\sigma_h$
from the definition~(\ref{eq:sigma_h}) when $h=0$.

When the lines $l_1$ and $l_{-1}$ are parallel, then 
there is room for only one infinite transnormal region
$\H=\R^2\backslash \left( \Pi_1 \cup \Pi_{-1} \right)$.
Recall that $\H$ is a stripe of length $h:=dist(l_1, l_{-1})$.
From (\ref{eq:Fh}), (\ref{eq:F1}) and (\ref{eq:F-1}) follows that
the union $\F_h \cup\F_{1} \cup  \F_{-1}$ forms a connected fiber $L_0$ 
if and only if $dist(p_1, p_{-1})=\sqrt{4+h^2}$.
Thus, $L_0$ is exactly a curve $\sigma_h$ 
from~(\ref{eq:sigma_h}) where $\{x_0, y_0\}=\{p_1, p_{-1}\}$ and $h>0$.

In both cases, by construction, the map $P$ is a signed distance to $L_0=\sigma_h$ 
(see Figure~\ref{fig:R2_fibers}). 

\begin{figure}[h]
    \centering
    \includegraphics[width=0.5\textwidth]{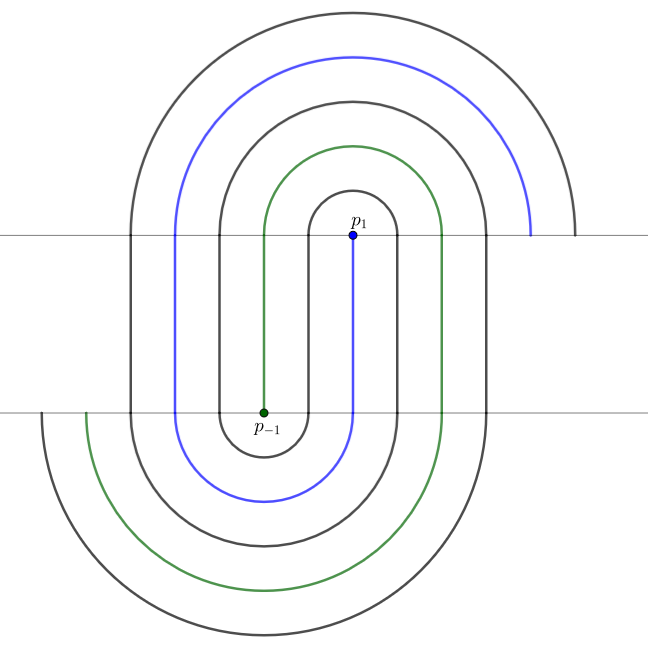}
    \caption{Decomposition of $\R^2$ into fibers: 
         $L_0$ (black curve), $L_1$ (blue curve), and $L_{-1}$ (green curve).}
    \label{fig:R2_fibers}
\end{figure}

\end{proof}


  \section{Proof of Theorem 2}\label{sec:sub_S2}

We will first remind the conditions of the theorem. 
Consider a two-dimensional unit sphere $\Sph^2$
and let $\Sph^2_1$ and $\Sph^2_{-1}$ be two closed hemispheres on $\Sph^2$
intersecting only along their boundary - the great circle $S^1$.
Let $a:=\pi/2k$ for $k>1$ and let $s$ be a number coprime with $k$.
Fix points $x_0, y_0 \in S_1$ at distance $2a\cdot s$ between each other. 
Consider two families of concentric half-circles 
$$\F_1 = \Sph^2_1\cap\{S(x_0, r_i)\}_{i=0}^{k-1}  \,\,\, \textnormal{and} \,\,\,
 \F_{-1} = \Sph^2_{-1}\cap\{S(y_0, r_i)\}_{i=0}^{k-1} $$ 
 where $S(x, r)$ are the points at distance $r$ from $x$ on $\Sph^2$ and 
 $r_i=(1+2i)\cdot a$ for any $i=0,1,..., k-1$.  
By choice of $x_0$ and $y_0$, the union 
\begin{equation}\label{eq:sigma_s}
    \sigma_{s}=\F_1 \cup \F_{-1} 
\end{equation}
forms a connected $C^{1,1}$ curve on $\Sph^2$.
\begin{theorem}\label{thm:S2_sub}
Let $P:\Sph^2\to Y$ be a submetry with connected leaves between
 a two-dimensional unite sphere~$\Sph^2$ and a one-dimensional Alexandrov space~$Y$.\\
Then $Y$ is a segment of  length $\pi/k$ and 
      \begin{enumerate}
          \item[a)]  If $k=1$, 
                then $P$ is a quotient map of an isometric group action of $S^1$ on $\Sph^2$.
          \item[b)] If $k>1$, then, up to isometry,
                $P$ is the signed distance function to the curve~$P^{-1}(0)=\sigma_{s}$ 
                for some $s$ coprime with $k$.
      \end{enumerate}
\end{theorem}
\begin{proof}
Since $\Sph^2$ is compact, the base space $Y$ must also be compact 
(see \cite[Lemma 2.3.]{KL20} and \cite[Proposition 1]{Ber87}).
Moreover, from Lemma~\ref{lem:smpl_con} we also know that $Y$ is simply connected. 
Therefore, $Y$ must be a segment $[-a, a]$.

By Lemma~\ref{lem:dist}, $P$ can be identified, up to isometry,
with the signed distance function to the middle fiber $L_0=P^{-1}(0)$.
By Lemma~\ref{Lem:1dim_pr}, 
$L_0$ is a $C^{1,1}$ curve without boundary.

To construct $L_0$, we consider the geodesic $\gamma(t):[0, t_0) \to \Sph^2$ 
emanating from $x\in L_0$ in the direction $h_x\in N_xL$.
We will write the geodesic in the opposite direction $-h_x\in N_xL$ 
 as $\gamma(-t)$, $t\geq 0$.
Note that for any horizontal geodesic the maximal time of its existence satisfies
\begin{equation}\label{eq:geod_time}
|t_0|\leq \pi.
\end{equation}

The projection $(P\circ\gamma)(t):[0, t_0) \to Y$ is a quasigeodesic in $Y$.
If $t_0>a$, then  by Proposition~\ref{quasigeod}, there is a partition of $[0, t_0)$  such that
\begin{equation}\label{eq:quasigeod_S2}
    (P\circ\gamma)\left(2ka+(-1)^k\e \right) =\e, \,\,\, \e \in [-a, a]
\end{equation}
when $k\in \N$ is such that $2ka+(-1)^k\e \leq t_0$.
In the opposite direction, the same holds replacing $t$ with $-t$ and $\e$ with $-\e$. 
In combination with (\ref{eq:geod_time}), it implies that 
\begin{equation}\label{eq:a_pi}
    a\leq \pi.
\end{equation}

Before proceeding with the construction of $L_0$, we first analyze the boundary fibers
$L_{+}=P^{-1}( a)$ and $L_{-}=P^{-1}(-a)$.
By Lemma \ref{Lem:1dim_pr} each of these fibers is either a point
or a  $C^{1,1}$ curve that may have boundary.

First, we observe that $L_{+}$ and $L_{-}$ must be compact.
Indeed, let $D_{\pm }$ be the set of all the points in $\Sph^2$ 
at distance $\leq a$ from $L_{\pm}$.
Then $\Sph^2 = D_{+} \cup D_-$ and
since $\Sph^2$ is compact, both $D_{+}$ and $D_-$ must be compact domains. 
Furthermore, neither $L_+$ nor $L_-$ can be a closed curve;
otherwise the boundary of $D_{\pm }$ splits into two components.
Since the boundary of $D_{\pm }$ is the fiber $L_0$, 
this contradicts its connectedness.
Therefore, each boundary fiber, $L_+$ and $L_-$,
is either a single point or a closed segment.\\
\par

\textit{Case 1.}
Suppose the boundary fiber $L_+$ is a single point $p_+$. 
Then the fiber $L_0$ is the set of points at distance $a$ from  $p_+$,
i.e.  $L_0$ is a circle of radius $a$ on the sphere $\Sph^2$. 
The points of the fiber $L_-$ lie at distance $a$ from $L_0$ 
in the direction  opposite from $L_+$ direction. 
Therefore, $a\leq \pi/2$.
If $a< \pi/2$, then $L_-$ is a circle of  radius $2a$ centered at $p_-$.
Since all fibers are connected, there are no other fibers on the sphere,
a contradiction. 
Thus, $2a = \pi$ and $L_-$ is also a single point $p_-$, antipodal to $p_+$.
Hence, $P$ is a quotient map of an isometric $S^1$-action on $\Sph^2$. 
This proves case a). \\
\par

\textit{Case 2.}
Suppose now the boundary fibers are closed segments:
$L_+ =[p_+, q_+]$ and $L_- = [p_-, q_-]$. 
Denote  
$$B:=\{p_+, p_-, q_+, q_-\} \,\,\, \textnormal{and} \,\,\, \V:=\Sph^2\backslash B.$$
By Lemma~\ref{lem:tran_region}, the open subset $\V$ is a transnormal region of $\Sph^2$.

Take one of the singular points, for example $p_+$.
The tangent space $T_{p_+}L_+$ contains only one unit tangent vector $\tau_+$. 
The normal space $N_{p_+}L_+$ is then a two-dimensional half-plane. 
The image of  $N_{p_+}L_+$ under the exponential map is a closed hemisphere 
$\Sph^2_+ \subset  \Sph^2$  bounded by the great circle~$S_+^1$.

The points of the fiber $L_0$ lie at distance $a$ from $p_+$, i.e.
they form a semicircle $s(p_+, a)$ on $\Sph^2_+$.
The geodesic $\gamma(t)$ starting at the point $x\in s(p_+, a)$ 
in the horizontal direction $h_x\in N_{x}L_0$ satisfies the following:
\begin{itemize}
    \item $\gamma(a) = p_+$;
    \item $\gamma(t) \subset \V $, for $t\in [0, a)$.
\end{itemize}

Moving in the opposite direction $-h_x\in N_{x}L_0$ along the horizontal geodesic,
from the equation~(\ref{eq:quasigeod_S2}) we can deduce, the following:
\begin{equation}\label{eq:fibersSph+}
    \gamma\left(-a-4ka \right) \in L_{-}; \,\,\,\,\,\, 
    \gamma\left(-2ak \right) \in L_{0}; \,\,\,\,\,\, 
   \gamma\left(-3a-4ka \right) \in L_{+}.    
\end{equation}
From another side, since  $\gamma(t)$ is on a sphere,
all geodesics starting at the point $x\in s(p_+, a)$ come together at the point $b=\gamma(a-\pi)$.
This means that $b$ lies outside of the transnormal region, and thus $b\in B$. 
From (\ref{eq:fibersSph+}) then follows that
\begin{itemize}
      \item there exists $k_0\in \N$ such that  $a=\pi/2k_0$ and
      \item  if $2k_0-1 \equiv 3 \pmod 4$, then $b = q_+$.
      \item  if $2k_0-1 \equiv 1 \pmod 4$, then $b = q_-$ (if $b=p_-$, we can swap the notations).
  \end{itemize}

We claim that the open hemisphere $\Sph^2_+ \backslash S_+^1$ is a transnormal region and 
does not have any singular points inside.
In fact, from (\ref{eq:fibersSph+}) we get that since  $x\in s(p_+, a)$ then for 
$k<k_0$ the points $\gamma\left(-a-4ka \right)$ form a semicircle 
$s\big(p_+, r_k\big)$  of the radius $r_k=(2+4k)\cdot a$ centered at $p_+$ in $\Sph^2_+$.
The points $\gamma\left(-3a-4ka \right)$ form a semicircle 
$s(p_+, r_k)$  of  radius $r_k=4(k+1)\cdot a$ in $\Sph^2_+$ (see Figure~\ref{fig:halfS2_fibers}).

This means that $\Sph^2_+ \backslash S_+^1$ intersects only 
the regular part of the boundary fibers $L_+$ and $L_-$
and therefore $\Sph^2_+ \backslash S_+^1 \subset \V$. 

The intersection of $\Sph^2_+$ with the middle fiber $L_0$ forms the family 
\begin{equation}\label{eq:F_+}
    \F_+=\Sph^2_+ \cap L_0 =\{s(p_+, r_k) \}_{k=0}^{k_0}, \,\,\, \textnormal{where} 
    \,\,\, r_k= \big(1+2k\big)\frac{\pi}{2k_0}.
\end{equation}

 \begin{figure}[h]
\centering \includegraphics[width=0.4\textwidth]{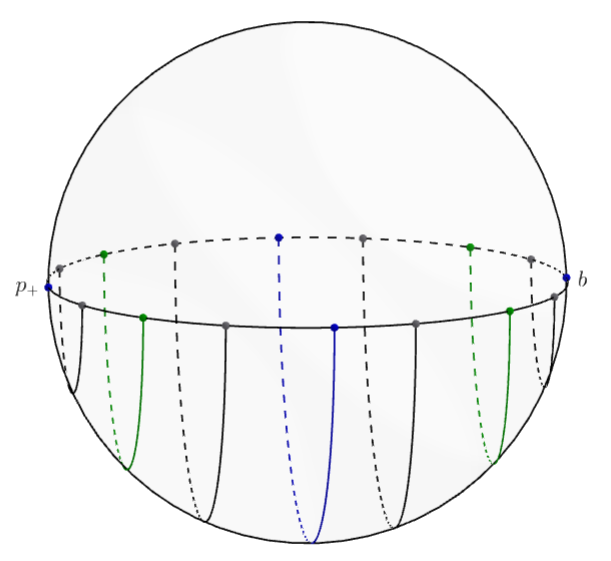}
\caption{Decomposition of $\Sph^2_+$ into fibers: 
    $L_0$ (black curve), $L_1$ (blue curve), and $L_{-1}$ (green curve).}
\label{fig:halfS2_fibers}
\end{figure}

We can repeat the construction for the remaining boundary points.
Namely, consider now the singular point $p_-\in B$.
The normal space $N_{p_-}L_-$ to the fiber $L_-$ at $p_-$ is a two-dimensional half-plane.
Its image under the exponential map is a closed hemisphere 
$\Sph^2_- \subset  \Sph^2$  bounded by the great circle~$S_-^1$.

Similarly to the above, the points of the fiber $L_0$ 
form a semicircle $s(p_-, a)$ of radius $a$ centered at $p_-$ on $\Sph^2_-$.
The horizontal geodesic $\gamma(t)$ to $s(p_-, a)$ at the point $x$ satisfies 
$\gamma(a) = p_-$ and  $\gamma(t) \subset \V$ for $t\in [0, a)$.
At the same time, all geodesics meet at the point $d=\gamma(a-\pi)$ for every $x\in s(p_-, a)$.
Since $a=\pi/2k_0$, from~(\ref{eq:quasigeod_S2}) we obtain, that 
\begin{itemize}
      \item $d=q_+$ if $2k_0-1 \equiv 1 \pmod 4$ 
      \item $d=q_-$ if $2k_0-1 \equiv 3 \pmod 4$.
  \end{itemize}

The intersection of $\Sph^2_-$ with the middle fiber $L_0$ forms the family 
\begin{equation}\label{eq:F_-}
    \F_-=\Sph^2_- \cap L_0 =\{s(p_-, r_k) \}_{k=0}^{k_0}, \,\,\, \textnormal{where} 
    \,\,\, r_k= \big(1+2k\big)\frac{\pi}{2k_0}.
\end{equation}

Moreover, the open hemisphere $\Sph^2_- \backslash S_-^1$ is also a transnormal region.
Indeed, from (\ref{eq:fibersSph+}) follows that for every $k<k_0$, 
the points $\gamma\left(-a-4ka \right) \in L_+$ and they form the semicircle 
$s\big(p_-, r_k\big)$ of the radius $r_k=(2+4k)\cdot a$ and center $p_-$ in $\Sph^2_-$.
The points $\gamma\left(-3a-4ka \right)\in L_-$ and form  
$s(p_-, r_k)$  of radius $r_k=4(k+1)\cdot a$ in $\Sph^2_-$.
Therefore,  $\Sph^2_-\backslash S_-^1$ do not contains points of $B$.

We claim that the boundary spheres $S^1_+$ and $S^1_-$ coincide. 
Otherwise, they intersect in two antipodal points and 
$S^1_-\cap (\Sph^2_+ \backslash S^1_+) \neq \emptyset$. 
Since the  points $p_-$ and $d$ are antipodal points on $S^1_-$
it implies that one of them lies in $\Sph^2_+ \backslash S^1_+$.
This leads to a contradiction to the fact that  $\Sph^2_+ \backslash S^1_+ \subset \V$.
Therefore,  $p_-$ and $d$ are points of the intersection of $S^1_+$ and $S^1_-$. 
In the same way, one can prove that $p_+$ and $b$ should also be the points of 
the intersection of $S^1_+$ and $S^1_-$.
Therefore, the boundary fibers $L+$ and $L_-$ intersect, resulting in a contradiction.

Hence, we have $S^1:=S^1_+=S^1_-$ that divides the ambient space $\Sph^2$ in two 
hemispheres $\Sph^2_+ $ and $\Sph^2_-$.
From~(\ref{eq:F_+}) and~(\ref{eq:F_-}), 
the  middle fiber $L_0$ is then given the union $\F_+\cup\F_-$.
It remains to determine a rotation of the hemisphere $\Sph^2_-$
such that the gluing along $S^1$ ensures that  $\F_+\cup\F_-$ is a connected curve.

In general, from~(\ref{eq:fibersSph+}) we know that 
the distance between $p_+$ and $p_-$ on the circle $S^1$ equals 
$$dist(p_+, p_-)=\big(2+4k \big)\frac{\pi}{2k_0} = \big(1+2k \big) \frac{\pi}{k_0}.$$
 Therefore, the union $\F_+\cup\F_-$ is a connected fiber $L_0$ if and only if
 $(1+2k)$ is coprime to $k_0$. 
Then $L_0$ is exactly a $C^{1,1}$ curve $\sigma_s$ from~(\ref{eq:sigma_s}),
where $s=1+2k$ and $a=\pi/2k_0$. 
By construction, the submetry $P$ is the signed distance to~$\sigma_s$
(see Figure~\ref{fig:S2_fibers}). 

 \begin{figure}[h]
\centering \includegraphics[width=0.5\textwidth]{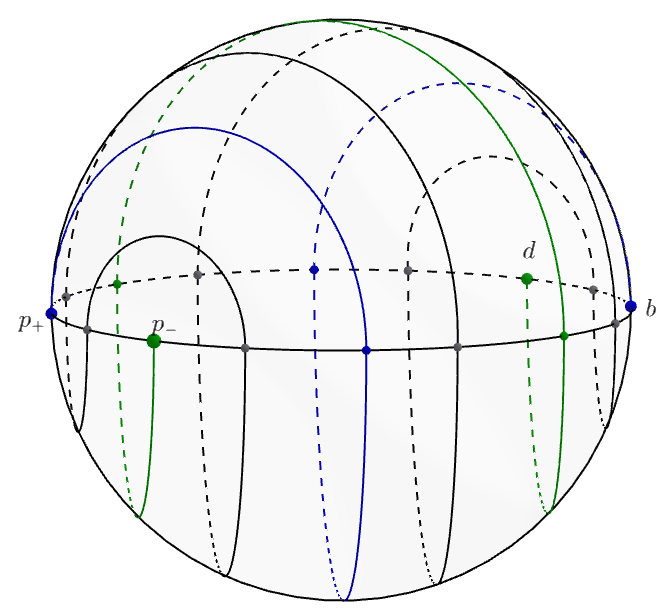}
\caption{Decomposition of $\Sph^2$ into fibers: 
    $L_0$ (black curve), $L_1$ (blue curve), and $L_{-1}$ (green curve).}
\label{fig:S2_fibers}
\end{figure}

\end{proof}

\newpage
\bibliographystyle{alpha}
\bibliography{bibliography}

\end{document}